\documentclass[preprint,12pt]{elsarticle}

\usepackage{amscd}
\usepackage{amsmath}
\usepackage{amsthm}
\usepackage{amsfonts}
\usepackage{graphicx}
\usepackage{amssymb}
\usepackage{color}
\usepackage{amsbsy}
\usepackage{graphicx}
\usepackage{amssymb,color,amsbsy}%, amsthm}
\usepackage{mathtools}
\usepackage{leftindex}

\usepackage{float}
\usepackage{comment}

\usepackage{stmaryrd}

\usepackage[ruled]{algorithm2e}

\usepackage[matrix,arrow]{xy}
\DeclareMathAlphabet{\mathpzc}{OT1}{pzc}{m}{it}

\usepackage{pgfpages}
\usepackage{tikz}
\usetikzlibrary{shapes.geometric}
\usetikzlibrary{decorations.pathmorphing}
\usetikzlibrary{arrows}
\usetikzlibrary{backgrounds}
\usetikzlibrary{positioning}
\usetikzlibrary{fit}
\usepackage{caption}
\usepackage{amsthm}

\usepackage{amsmath}

\usepackage[pagebackref=true, bookmarksopen=true,colorlinks=true, linkcolor=red,citecolor=blue]{hyperref}

\usepackage[capitalise]{cleveref}  %References will output ‘Theorem 3.21’, without you having to enter the “Theorem.”

\global\long\def\ii{\cap}%

\global\long\def\u{\cup}%

\global\long\def\I{\bigcap}%

\global\long\def\U{\bigcup}%

\global\long\def\s{\subset}%

\global\long\def\iso{\approx}%

\global\long\def\ud{\mathbin{\dot{\u}}}%

\global\long\def\p{\prime}%

\global\long\def\P{\prime}%

\global\long\def\ñ{\sim}%

\global\long\def\core#1{\textnormal{core}(#1)}%

\global\long\def\corona#1{\textnormal{corona}(#1)}%

\global\long\def\a#1{\left|#1\right|}%

\usepackage{tikz-cd}

\usepackage{tkz-graph}
\usepackage{multicol}

\usepackage{subcaption}

\newtheorem{theorem}{Theorem}[section]

\newtheorem{claim}[theorem]{Claim}

\newtheorem{conjecture}[theorem]{Conjecture}
\newtheorem{corollary}[theorem]{Corollary}

\newtheorem{lemma}[theorem]{Lemma}

\newtheorem{problem}[theorem]{Problem}
\newtheorem{proposition}[theorem]{Proposition}

\newtheorem{observation}[theorem]{Observation}
\usepackage{comment}
\begin{document}

\begin{abstract}
Let $\core G$ and $\corona G$ denote the intersection and the union,
respectively, of all maximum independent sets of a graph $G$.
A graph is called \emph{$2$-bicritical} if $\a{N(S)}>\a S$ for every nonempty
independent set $S$. Pulleyblank 1979 showed that almost all graphs are
$2$-bicritical.

In this paper, we study the structure of maximum independent sets in
$2$-bicritical graphs with at most two odd cycles.
Using ear--pendant decompositions, we obtain a complete structural
classification of these graphs into four families: one-odd cycle,
fused-odd, even-linked, and odd-linked graphs.
For each family, we compute explicitly $\alpha(G)$, $\core G$, and
$\corona G$, and describe the corresponding matching structure.

We prove that $\a{\core G}+\a{\corona G}$ equals either $2\alpha(G),2\alpha(G)+1$ or
$2\alpha(G)+2$, and we give a complete, purely structural characterization
of the graphs in each case in terms of the relative position of their odd
cycles.

	These results extend a theory originally developed for
	K\"onig--Egerv\'ary graphs and later for almost bipartite graphs
	to a broader non-K\"onig--Egerv\'ary setting.
\end{abstract}

\begin{frontmatter} 
	\title{Core and Corona in 2-Bicritical Odd-Bicyclic Graphs}

	%Core and Corona in Odd-Bicyclic Graphs

	%\author[IMASL,DEPTO]{XXXXXXXXXXXXXXXX} 	\ead{XXXXXXXXXXXX@unsl.edu.ar} 
	
	\author[IMASL,DEPTO]{Kevin Pereyra} 	\ead{kdpereyra@unsl.edu.ar}

	%\author[]{XXX} 	%\ead{xxx@XXX} 
	%\author[]{XXX} 	%\ead{xxx@XXX}

	\address[IMASL]{Instituto de Matem\'atica Aplicada San Luis, Universidad Nacional de San Luis and CONICET, San Luis, Argentina.}
	\address[DEPTO]{Departamento de Matem\'atica, Universidad Nacional de San Luis, San Luis, Argentina.} 	
	
	\date{Received: date / Accepted: date} 
	
\end{frontmatter} %

\section{Introduction}\label{sss0}

Let $\alpha(G)$ denote the cardinality of a maximum independent set,
and let $\mu(G)$ be the size of a maximum matching in $G=(V,E)$.
It is known that $\alpha(G)+\mu(G)$ equals the order of $G$,
in which case $G$ is a König--Egerváry graph 
\cite{deming1979independence,gavril1977testing,sterboul1979characterization}.
K\H{o}nig-Egerv\'{a}ry graphs have been extensively studied
\cite{bourjolly2009node,jarden2017two,levit2006alpha,levit2012critical,jaume5404413kr}.
It is known that every bipartite graph is a König--Egerváry graph 
\cite{egervary1931combinatorial}. These graphs were independently introduced by Deming \cite{deming1979independence}, Sterboul \cite{sterboul1979characterization}, and \cite{gavril1977testing}. A graph $G$ is almost bipartite if it has only one odd cycle.
Let $\Omega^{*}(G)=\left\{ S:S\textnormal{ is an independent set of }G\right\}$,
$\Omega(G)=\{S:S$ is a maximum independent set of $G\}$,
$\textnormal{core}(G)=\I\left\{ S:S\in\Omega(G)\right\}$ 
\cite{levit2003alpha+}, and 
$\textnormal{corona}(G)=\U\left\{ S:S\in\Omega(G)\right\}$ 
\cite{boros2002number}.

\begin{theorem}If $G$ is an König--Egerváry graph, then 
		\begin{itemize}
		\item \textnormal{\cite{levit2011set}} $\left|\textnormal{corona}(G)\right|+\left|\textnormal{core}(G)\right|=2\alpha(G),$
		\item \textnormal{\cite{levit2003alpha+}} $\textnormal{corona}(G)\cup N(\textnormal{core}(G))=V(G).$
	\end{itemize}
\end{theorem}

\begin{theorem}
	[\label{levit}\cite{levit2025almost}] If $G$ is an almost bipartite
	non-König--Egerváry graph, then 
	\begin{itemize}
		\item $\left|\textnormal{corona}(G)\right|+\left|\textnormal{core}(G)\right|=2\alpha(G)+1,$
		\item $\textnormal{corona}(G)\cup N(\textnormal{core}(G))=V(G).$
	\end{itemize}
\end{theorem}

From a structural point of view, graphs with at most two odd cycles
constitute the first nontrivial extension beyond the almost bipartite case.
While bipartite graphs and graphs with a single odd cycle exhibit a rather
rigid behavior with respect to maximum independent sets, the presence of two
odd cycles allows for qualitatively different configurations, depending on
their relative position inside the graph.

The first item of \cref{levit} motivates \cref{mainproblem}.

\begin{problem}[\cite{levit2025almost}\label{mainproblem}]
	Characterize graphs enjoying $\left|\textnormal{corona}(G)\right|+\left|\textnormal{core}(G)\right|=2\alpha(G)+1,$
\end{problem}

In \cite{kevincoronacore}, a reductive-type characterization was obtained
for graphs satisfying
$\a{\core G}+\a{\corona G}=2\alpha(G)+1$.
In this paper, we compute $\core G$ and $\corona G$ explicitly
for $2$-bicritical graphs with at most two odd cycles.
As a consequence, we show that
$\a{\core G}+\a{\corona G}$ is equal to either $2\alpha(G)$ or
$2\alpha(G)+1$, and we characterize the graphs in each case.

\medskip

The paper is organized as follows.
In \cref{sss0} we present the general context of the problem and introduce
the fundamental concepts.
In \cref{sss1} we fix the notation used throughout the paper.
In \cref{sss2} we classify $2$-bicritical graphs with at most two odd cycles
into four distinct families: one-odd cycle, fused-odd, even-linked, and
odd-linked graphs, according to the behavior of their maximum independent
sets, and we derive structural information for each class.
In \cref{sss3} we study even-linked graphs, in \cref{sss4} we study
odd-linked graphs, and in \cref{sss5} we study fused-odd and one-odd cycle
graphs.
Finally, in \cref{sss6} we obtain general characterizations and propose
some conjectures and open problems.

\section{Preliminaries}\label{sss1}
All graphs considered in this paper are finite, undirected, and simple. 
For any undefined terminology or notation, we refer the reader to 
Lovász and Plummer \cite{LP} or Diestel \cite{Distel}.

Let \( G = (V, E) \) be a simple graph, where \( V = V(G) \) is the finite set of vertices and \( E = E(G) \) is the set of edges, with \( E \subseteq \{\{u, v\} : u, v \in V, u \neq v\} \). We denote the edge \( e=\{u, v\} \) as \( uv \). A subgraph of \( G \) is a graph \( H \) such that \( V(H) \subseteq V(G) \) and \( E(H) \subseteq E(G) \). A subgraph \( H \) of \( G \) is called a \textit{spanning} subgraph if \( V(H) = V(G) \). 

Let \( e \in E(G) \) and \( v \in V(G) \). We define \( G - e := (V, E - \{e\}) \) and \( G - v := (V - \{v\}, \{uw \in E : u,w \neq v\}) \). If \( X \subseteq V(G) \), the \textit{induced} subgraph of \( G \) by \( X \) is the subgraph \( G[X]=(X,F) \), where \( F:=\{uv \!\in\! E(G) : u, v \!\in \! X\} \).

The number of vertices in a graph $G$ is called the \textit{order} of the graph and denoted by $\left|G\right|$ or $n(G)$.
A \textit{cycle} in $G$ is called \textit{odd} (resp. \textit{even}) if it has an odd (resp. even) number of edges.

For a vertex $v\in V(G)$, the \emph{neighborhood} of $v$ is
\[
N_G(v)=\{u\in V(G): uv\in E(G)\}.
\]
When no confusion arises, we write $N(v)$ instead of $N_G(v)$. For a set $S\subseteq V(G)$, the \emph{neighborhood} of $S$ is
\[
N_G(S)=\bigcup_{v\in S} N_G(v).
\]

\begin{comment}
	\noindent The \emph{degree} of a vertex $v\in V(G)$ is
$
\deg_G(v)=|N_G(v)|.
$
\noindent The \emph{minimum degree} of $G$ is
$
\delta(G)=\min\{\deg_G(v): v\in V(G)\}.
$
A graph $G$ is called \emph{$r$-regular} if $\deg_G(v)=r$ for every
$v\in V(G)$. A vertex $v\in V(G)$ is called an \emph{isolated vertex} if
$
\deg_G(v)=0.
$
The number of isolated vertices of a graph $G$ is denoted by $i(G)$.
\end{comment}

A \textit{matching} \(M\) in a graph \(G\) is a set of pairwise non-adjacent edges. The \textit{matching number} of \(G\), denoted by  \(\mu(G)\), is the maximum cardinality of any matching in \(G\). Matchings induce an involution on the vertex set of the graph: \(M:V(G)\rightarrow V(G)\), where \(M(v)=u\) if \(uv \in M\), and \(M(v)=v\) otherwise. If \(S, U \subseteq V(G)\) with \(S \cap U = \emptyset\), we say that \(M\) is a matching from \(S\) to \(U\) if \(M(S) \subseteq U\). A matching $M$ is \emph{perfect} if $M(v)\neq v$ for every vertex
of the graph. A matching is \emph{near-perfect} if \( \left|{v \in V(G) : M(v) = v}\right| = 1 \).  A graph is a factor-critical graph if $G-v$ has perfect matching
for every vertex $v\in V(G)$.

A vertex set \( S \subseteq V \) is \textit{independent} if, for every pair of vertices \( u, v \in S \), we have \( uv \notin E \). 
The number of vertices in a maximum independent set is denoted by \( \alpha(G) \).  A \textit{bipartite} graph is a graph whose vertex set can be partitioned into two disjoint independent sets.

We will use the Gallai–Edmonds structure theorem throughout the paper, together with the associated notation.

\begin{theorem}
	[\label{ge}\cite{edmonds1965paths,gallai1964maximale} Gallai--Edmonds
	structure theorem]
	Let $G$ be a graph, and define
	\begin{align*}
		D(G) & := \{ v : \textnormal{there exists a maximum matching that misses } v \}, \\
		A(G) & := \{ v : v \textnormal{ is adjacent to some } u \in D(G), \textnormal{ but } v \notin D(G) \}, \\
		C(G) & := V(G) - (D(G) \cup A(G)).
	\end{align*}
	
	\noindent If $G_{1}, \dots, G_{k}$ are the connected components of $G[D(G)]$
	and $M$ is a maximum matching of $G$, then:
	\begin{enumerate}
		\item $M$ covers $C(G)$ and matches $A(G)$ into distinct components of $G[D(G)]$.
		\item Each $G_i$ is a factor-critical graph, and the restriction of $M$ to $G_i$ 
		is a near-perfect matching.
	\end{enumerate}
\end{theorem}

\section{Structural classification of odd-bicyclic graphs}\label{sss2}

The aim of this section is to isolate all possible structural configurations
of $2$-bicritical graphs with at most two odd cycles.
This classification will be used systematically in the subsequent sections
to compute $\alpha(G)$, $\core G$, and $\corona G$ explicitly.

\medskip

The notion of 2-bicritical graphs was introduced
in \cite{pulleyblank1979minimum}, and they can be characterized as follows.

\begin{theorem}[\cite{pulleyblank1979minimum}\label{1928u3123}]
	A  graph $G$ is $2$-bicritical if and only if $\left|N(S)\right|>\left|S\right|$
	for every nonempty independent set $S\subseteq V(G)$.
\end{theorem}

The class of $2$-bicritical graphs can be regarded as the structural counterpart of König--Egerváry graphs \cite{larson2011critical}. In recent works, several new properties of $2$-bicritical graphs have been established; see, for instance, \cite{kevinSDKECHAR, kevinSDKEGE, kevinPOSYFACTOR}. It is important to note that Pulleyblank in 1979 showed that almost all graphs are 2-bicritical \cite{pulleyblank1979minimum}. In \cite{larson2011critical}, Larson showed that every graph can be decomposed into a 2-bicritical graph and a Kőnig–Egerváry graph.

\medskip

We say that $G^{\P}$ is an odd homeomorph of $G$ if $G^{\P}$ is
obtained by replacing each edge of $G$ with a path of odd length,
such that these paths remain internally disjoint. Let $B$ be a subgraph
of $G$. An ear in $G$ with respect to $B$ is an odd-length path
in $G$ whose endpoints lie in $B$, but whose internal vertices
lie outside $B$, and in which all internal vertices are distinct.
A pendant in $G$ with respect to $B$ consists of an odd-length simple cycle
$C$ in $G$, vertex-disjoint from $B$, together with a positive-length
simple path with one end in $C$ and the other in $B$, with all other
vertices lying outside both $B$ and $C$. The vertex in $B$ is called
the end of the pendant.

An \emph{ear-pendant} decomposition of a graph $G$ is a sequence
$G_{0},G_{1},\dots,G_{p}=G$ of graphs where $G_{0}$ is either an
odd cycle or an odd homeomorph of $K_{4}$, and for each $i\in\{1,\dots,p\}$,
$G_{i}$ is obtained from $G_{i-1}$ by adding either an ear or a
pendant. In \cite{bourjolly1989konig} the following characterization of 2-bicritical
graphs in terms of ear-pendant decompositions was given.

\begin{theorem}[\label{puyelarpendietnedesc}\cite{bourjolly1989konig}]
	A connected graph is 2-bicritical if and only if it has an ear-pendant decomposition.
\end{theorem}

\begin{comment}
	
Un grafo es un \emph{odd-bicyclic} graph si tiene exactamente dos
ciclos impares. En esta sección presentamos una clasificación estructural
de los grafos odd-bicyclic en cinco tipos distintos. 

Para desarrollar esta clasificación, comenzamos introduciendo algunas
definiciones y notación que serán utilizadas a lo largo de todo el
trabajo.

\end{comment}

Let $G$ be a connected $2$-bicritical graph with at most two odd cycles, and note that
$G$ contains at least one odd cycle.
Then $G$ can be classified into one of the following four classes.
This classification reflects qualitatively different behaviors of the
parameters $\core G,\corona G,\alpha(G)$.

\begin{itemize}
	\item If $G$ has a unique odd cycle, we say that $G$ is a
	\emph{one-odd cycle} graph.
	In this case, by the ear-pendant decomposition, it is easy to see that
	$G$ necessarily coincides with this odd cycle (see \cref{asidj123llllk}).
\end{itemize}

\noindent Suppose now that $G$ contains exactly two odd cycles and consider
an ear-pendant decomposition of $G$:
\[
G_{0},\dots,G_{p}=G.
\]

\noindent If $G_{0}$ is an odd homeomorph of $K_{4}$, then $G$ has more than
two odd cycles, a contradiction.
Therefore, $G_{0}$ is an odd cycle.
\begin{itemize}
	\item If $G_{1}$ is obtained from $G_{0}$ by adding an odd ear, then
	$p=1$.
	Indeed, in this case, adding either an odd ear or a pendant to $G_{1}$
	generates at least one new odd cycle (see \cref{asidj123llllk}).
	In this case, we say that $G$ is a \emph{fused-odd} graph.
\end{itemize}

\begin{figure}[H]
	
	\begin{center}

		\tikzset{every picture/.style={line width=0.75pt}} %set default line width to 0.75pt        
		
		\begin{tikzpicture}[x=0.75pt,y=0.75pt,yscale=-1,xscale=1]
			
			%Straight Lines [id:da5684062955632688] 
			\draw    (268.11,144) -- (211.76,144) ;
			\draw [shift={(211.76,144)}, rotate = 180] [color={rgb, 255:red, 0; green, 0; blue, 0 }  ][fill={rgb, 255:red, 0; green, 0; blue, 0 }  ][line width=0.75]      (0, 0) circle [x radius= 3.35, y radius= 3.35]   ;
			\draw [shift={(268.11,144)}, rotate = 180] [color={rgb, 255:red, 0; green, 0; blue, 0 }  ][fill={rgb, 255:red, 0; green, 0; blue, 0 }  ][line width=0.75]      (0, 0) circle [x radius= 3.35, y radius= 3.35]   ;
			%Straight Lines [id:da842150677856865] 
			\draw    (268.11,144) -- (285.52,90.41) ;
			\draw [shift={(285.52,90.41)}, rotate = 288] [color={rgb, 255:red, 0; green, 0; blue, 0 }  ][fill={rgb, 255:red, 0; green, 0; blue, 0 }  ][line width=0.75]      (0, 0) circle [x radius= 3.35, y radius= 3.35]   ;
			\draw [shift={(268.11,144)}, rotate = 288] [color={rgb, 255:red, 0; green, 0; blue, 0 }  ][fill={rgb, 255:red, 0; green, 0; blue, 0 }  ][line width=0.75]      (0, 0) circle [x radius= 3.35, y radius= 3.35]   ;
			%Straight Lines [id:da28386711909902496] 
			\draw    (239.93,57.29) -- (285.52,90.41) ;
			\draw [shift={(285.52,90.41)}, rotate = 36] [color={rgb, 255:red, 0; green, 0; blue, 0 }  ][fill={rgb, 255:red, 0; green, 0; blue, 0 }  ][line width=0.75]      (0, 0) circle [x radius= 3.35, y radius= 3.35]   ;
			\draw [shift={(239.93,57.29)}, rotate = 36] [color={rgb, 255:red, 0; green, 0; blue, 0 }  ][fill={rgb, 255:red, 0; green, 0; blue, 0 }  ][line width=0.75]      (0, 0) circle [x radius= 3.35, y radius= 3.35]   ;
			%Straight Lines [id:da47632787023862067] 
			\draw    (239.93,57.29) -- (194.35,90.41) ;
			\draw [shift={(194.35,90.41)}, rotate = 144] [color={rgb, 255:red, 0; green, 0; blue, 0 }  ][fill={rgb, 255:red, 0; green, 0; blue, 0 }  ][line width=0.75]      (0, 0) circle [x radius= 3.35, y radius= 3.35]   ;
			\draw [shift={(239.93,57.29)}, rotate = 144] [color={rgb, 255:red, 0; green, 0; blue, 0 }  ][fill={rgb, 255:red, 0; green, 0; blue, 0 }  ][line width=0.75]      (0, 0) circle [x radius= 3.35, y radius= 3.35]   ;
			%Straight Lines [id:da03497756630368143] 
			\draw    (211.76,144) -- (194.35,90.41) ;
			\draw [shift={(194.35,90.41)}, rotate = 252] [color={rgb, 255:red, 0; green, 0; blue, 0 }  ][fill={rgb, 255:red, 0; green, 0; blue, 0 }  ][line width=0.75]      (0, 0) circle [x radius= 3.35, y radius= 3.35]   ;
			\draw [shift={(211.76,144)}, rotate = 252] [color={rgb, 255:red, 0; green, 0; blue, 0 }  ][fill={rgb, 255:red, 0; green, 0; blue, 0 }  ][line width=0.75]      (0, 0) circle [x radius= 3.35, y radius= 3.35]   ;
			%Straight Lines [id:da5954720999171962] 
			\draw    (268.11,144) ;
			\draw [shift={(268.11,144)}, rotate = 0] [color={rgb, 255:red, 0; green, 0; blue, 0 }  ][fill={rgb, 255:red, 0; green, 0; blue, 0 }  ][line width=0.75]      (0, 0) circle [x radius= 3.35, y radius= 3.35]   ;
			\draw [shift={(268.11,144)}, rotate = 0] [color={rgb, 255:red, 0; green, 0; blue, 0 }  ][fill={rgb, 255:red, 0; green, 0; blue, 0 }  ][line width=0.75]      (0, 0) circle [x radius= 3.35, y radius= 3.35]   ;
			%Straight Lines [id:da37200074677700923] 
			\draw    (211.76,144) ;
			\draw [shift={(211.76,144)}, rotate = 0] [color={rgb, 255:red, 0; green, 0; blue, 0 }  ][fill={rgb, 255:red, 0; green, 0; blue, 0 }  ][line width=0.75]      (0, 0) circle [x radius= 3.35, y radius= 3.35]   ;
			\draw [shift={(211.76,144)}, rotate = 0] [color={rgb, 255:red, 0; green, 0; blue, 0 }  ][fill={rgb, 255:red, 0; green, 0; blue, 0 }  ][line width=0.75]      (0, 0) circle [x radius= 3.35, y radius= 3.35]   ;
			%Straight Lines [id:da44961009369815985] 
			\draw    (194.35,90.41) ;
			\draw [shift={(194.35,90.41)}, rotate = 0] [color={rgb, 255:red, 0; green, 0; blue, 0 }  ][fill={rgb, 255:red, 0; green, 0; blue, 0 }  ][line width=0.75]      (0, 0) circle [x radius= 3.35, y radius= 3.35]   ;
			\draw [shift={(194.35,90.41)}, rotate = 0] [color={rgb, 255:red, 0; green, 0; blue, 0 }  ][fill={rgb, 255:red, 0; green, 0; blue, 0 }  ][line width=0.75]      (0, 0) circle [x radius= 3.35, y radius= 3.35]   ;
			%Straight Lines [id:da13165331673725245] 
			\draw    (285.52,90.41) ;
			\draw [shift={(285.52,90.41)}, rotate = 0] [color={rgb, 255:red, 0; green, 0; blue, 0 }  ][fill={rgb, 255:red, 0; green, 0; blue, 0 }  ][line width=0.75]      (0, 0) circle [x radius= 3.35, y radius= 3.35]   ;
			\draw [shift={(285.52,90.41)}, rotate = 0] [color={rgb, 255:red, 0; green, 0; blue, 0 }  ][fill={rgb, 255:red, 0; green, 0; blue, 0 }  ][line width=0.75]      (0, 0) circle [x radius= 3.35, y radius= 3.35]   ;
			%Straight Lines [id:da5240130386890922] 
			\draw    (239.93,57.29) ;
			\draw [shift={(239.93,57.29)}, rotate = 0] [color={rgb, 255:red, 0; green, 0; blue, 0 }  ][fill={rgb, 255:red, 0; green, 0; blue, 0 }  ][line width=0.75]      (0, 0) circle [x radius= 3.35, y radius= 3.35]   ;
			\draw [shift={(239.93,57.29)}, rotate = 0] [color={rgb, 255:red, 0; green, 0; blue, 0 }  ][fill={rgb, 255:red, 0; green, 0; blue, 0 }  ][line width=0.75]      (0, 0) circle [x radius= 3.35, y radius= 3.35]   ;
			%Straight Lines [id:da4299024282363989] 
			\draw    (355.55,144.95) -- (337.95,91.42) ;
			\draw [shift={(337.95,91.42)}, rotate = 251.8] [color={rgb, 255:red, 0; green, 0; blue, 0 }  ][fill={rgb, 255:red, 0; green, 0; blue, 0 }  ][line width=0.75]      (0, 0) circle [x radius= 3.35, y radius= 3.35]   ;
			\draw [shift={(355.55,144.95)}, rotate = 251.8] [color={rgb, 255:red, 0; green, 0; blue, 0 }  ][fill={rgb, 255:red, 0; green, 0; blue, 0 }  ][line width=0.75]      (0, 0) circle [x radius= 3.35, y radius= 3.35]   ;
			%Straight Lines [id:da38917464240578403] 
			\draw    (355.55,144.95) -- (411.9,144.75) ;
			\draw [shift={(411.9,144.75)}, rotate = 359.8] [color={rgb, 255:red, 0; green, 0; blue, 0 }  ][fill={rgb, 255:red, 0; green, 0; blue, 0 }  ][line width=0.75]      (0, 0) circle [x radius= 3.35, y radius= 3.35]   ;
			\draw [shift={(355.55,144.95)}, rotate = 359.8] [color={rgb, 255:red, 0; green, 0; blue, 0 }  ][fill={rgb, 255:red, 0; green, 0; blue, 0 }  ][line width=0.75]      (0, 0) circle [x radius= 3.35, y radius= 3.35]   ;
			%Straight Lines [id:da22826561558253877] 
			\draw    (429.12,91.1) -- (421.11,116.05) -- (411.9,144.75) ;
			\draw [shift={(411.9,144.75)}, rotate = 107.8] [color={rgb, 255:red, 0; green, 0; blue, 0 }  ][fill={rgb, 255:red, 0; green, 0; blue, 0 }  ][line width=0.75]      (0, 0) circle [x radius= 3.35, y radius= 3.35]   ;
			\draw [shift={(429.12,91.1)}, rotate = 107.8] [color={rgb, 255:red, 0; green, 0; blue, 0 }  ][fill={rgb, 255:red, 0; green, 0; blue, 0 }  ][line width=0.75]      (0, 0) circle [x radius= 3.35, y radius= 3.35]   ;
			%Straight Lines [id:da07499411904529663] 
			\draw    (429.12,91.1) -- (383.42,58.14) ;
			\draw [shift={(383.42,58.14)}, rotate = 215.8] [color={rgb, 255:red, 0; green, 0; blue, 0 }  ][fill={rgb, 255:red, 0; green, 0; blue, 0 }  ][line width=0.75]      (0, 0) circle [x radius= 3.35, y radius= 3.35]   ;
			\draw [shift={(429.12,91.1)}, rotate = 215.8] [color={rgb, 255:red, 0; green, 0; blue, 0 }  ][fill={rgb, 255:red, 0; green, 0; blue, 0 }  ][line width=0.75]      (0, 0) circle [x radius= 3.35, y radius= 3.35]   ;
			%Straight Lines [id:da8488731386815515] 
			\draw    (337.95,91.42) -- (383.42,58.14) ;
			\draw [shift={(383.42,58.14)}, rotate = 323.8] [color={rgb, 255:red, 0; green, 0; blue, 0 }  ][fill={rgb, 255:red, 0; green, 0; blue, 0 }  ][line width=0.75]      (0, 0) circle [x radius= 3.35, y radius= 3.35]   ;
			\draw [shift={(337.95,91.42)}, rotate = 323.8] [color={rgb, 255:red, 0; green, 0; blue, 0 }  ][fill={rgb, 255:red, 0; green, 0; blue, 0 }  ][line width=0.75]      (0, 0) circle [x radius= 3.35, y radius= 3.35]   ;
			%Straight Lines [id:da19723470724253833] 
			\draw    (355.55,144.95) ;
			\draw [shift={(355.55,144.95)}, rotate = 0] [color={rgb, 255:red, 0; green, 0; blue, 0 }  ][fill={rgb, 255:red, 0; green, 0; blue, 0 }  ][line width=0.75]      (0, 0) circle [x radius= 3.35, y radius= 3.35]   ;
			\draw [shift={(355.55,144.95)}, rotate = 0] [color={rgb, 255:red, 0; green, 0; blue, 0 }  ][fill={rgb, 255:red, 0; green, 0; blue, 0 }  ][line width=0.75]      (0, 0) circle [x radius= 3.35, y radius= 3.35]   ;
			%Straight Lines [id:da16495595077777947] 
			\draw    (337.95,91.42) ;
			\draw [shift={(337.95,91.42)}, rotate = 0] [color={rgb, 255:red, 0; green, 0; blue, 0 }  ][fill={rgb, 255:red, 0; green, 0; blue, 0 }  ][line width=0.75]      (0, 0) circle [x radius= 3.35, y radius= 3.35]   ;
			\draw [shift={(337.95,91.42)}, rotate = 0] [color={rgb, 255:red, 0; green, 0; blue, 0 }  ][fill={rgb, 255:red, 0; green, 0; blue, 0 }  ][line width=0.75]      (0, 0) circle [x radius= 3.35, y radius= 3.35]   ;
			%Straight Lines [id:da4010373859386017] 
			\draw    (383.42,58.14) ;
			\draw [shift={(383.42,58.14)}, rotate = 0] [color={rgb, 255:red, 0; green, 0; blue, 0 }  ][fill={rgb, 255:red, 0; green, 0; blue, 0 }  ][line width=0.75]      (0, 0) circle [x radius= 3.35, y radius= 3.35]   ;
			\draw [shift={(383.42,58.14)}, rotate = 0] [color={rgb, 255:red, 0; green, 0; blue, 0 }  ][fill={rgb, 255:red, 0; green, 0; blue, 0 }  ][line width=0.75]      (0, 0) circle [x radius= 3.35, y radius= 3.35]   ;
			%Straight Lines [id:da956450965441464] 
			\draw    (411.9,144.75) ;
			\draw [shift={(411.9,144.75)}, rotate = 0] [color={rgb, 255:red, 0; green, 0; blue, 0 }  ][fill={rgb, 255:red, 0; green, 0; blue, 0 }  ][line width=0.75]      (0, 0) circle [x radius= 3.35, y radius= 3.35]   ;
			\draw [shift={(411.9,144.75)}, rotate = 0] [color={rgb, 255:red, 0; green, 0; blue, 0 }  ][fill={rgb, 255:red, 0; green, 0; blue, 0 }  ][line width=0.75]      (0, 0) circle [x radius= 3.35, y radius= 3.35]   ;
			%Straight Lines [id:da2216508758459962] 
			\draw    (429.12,91.1) ;
			\draw [shift={(429.12,91.1)}, rotate = 0] [color={rgb, 255:red, 0; green, 0; blue, 0 }  ][fill={rgb, 255:red, 0; green, 0; blue, 0 }  ][line width=0.75]      (0, 0) circle [x radius= 3.35, y radius= 3.35]   ;
			\draw [shift={(429.12,91.1)}, rotate = 0] [color={rgb, 255:red, 0; green, 0; blue, 0 }  ][fill={rgb, 255:red, 0; green, 0; blue, 0 }  ][line width=0.75]      (0, 0) circle [x radius= 3.35, y radius= 3.35]   ;
			%Straight Lines [id:da23671545010024742] 
			\draw    (386.43,90.96) -- (383.42,58.14) ;
			\draw [shift={(383.42,58.14)}, rotate = 264.76] [color={rgb, 255:red, 0; green, 0; blue, 0 }  ][fill={rgb, 255:red, 0; green, 0; blue, 0 }  ][line width=0.75]      (0, 0) circle [x radius= 3.35, y radius= 3.35]   ;
			\draw [shift={(386.43,90.96)}, rotate = 264.76] [color={rgb, 255:red, 0; green, 0; blue, 0 }  ][fill={rgb, 255:red, 0; green, 0; blue, 0 }  ][line width=0.75]      (0, 0) circle [x radius= 3.35, y radius= 3.35]   ;
			%Straight Lines [id:da9664789189534272] 
			\draw    (355.55,144.95) -- (378.43,124.96) ;
			\draw [shift={(378.43,124.96)}, rotate = 318.86] [color={rgb, 255:red, 0; green, 0; blue, 0 }  ][fill={rgb, 255:red, 0; green, 0; blue, 0 }  ][line width=0.75]      (0, 0) circle [x radius= 3.35, y radius= 3.35]   ;
			\draw [shift={(355.55,144.95)}, rotate = 318.86] [color={rgb, 255:red, 0; green, 0; blue, 0 }  ][fill={rgb, 255:red, 0; green, 0; blue, 0 }  ][line width=0.75]      (0, 0) circle [x radius= 3.35, y radius= 3.35]   ;
			%Straight Lines [id:da4521437057170371] 
			\draw    (378.43,124.96) -- (386.43,90.96) ;
			\draw [shift={(386.43,90.96)}, rotate = 283.24] [color={rgb, 255:red, 0; green, 0; blue, 0 }  ][fill={rgb, 255:red, 0; green, 0; blue, 0 }  ][line width=0.75]      (0, 0) circle [x radius= 3.35, y radius= 3.35]   ;
			\draw [shift={(378.43,124.96)}, rotate = 283.24] [color={rgb, 255:red, 0; green, 0; blue, 0 }  ][fill={rgb, 255:red, 0; green, 0; blue, 0 }  ][line width=0.75]      (0, 0) circle [x radius= 3.35, y radius= 3.35]   ;
			%Straight Lines [id:da30746948088812354] 
			\draw    (489.55,145.95) -- (471.95,92.42) ;
			\draw [shift={(471.95,92.42)}, rotate = 251.8] [color={rgb, 255:red, 0; green, 0; blue, 0 }  ][fill={rgb, 255:red, 0; green, 0; blue, 0 }  ][line width=0.75]      (0, 0) circle [x radius= 3.35, y radius= 3.35]   ;
			\draw [shift={(489.55,145.95)}, rotate = 251.8] [color={rgb, 255:red, 0; green, 0; blue, 0 }  ][fill={rgb, 255:red, 0; green, 0; blue, 0 }  ][line width=0.75]      (0, 0) circle [x radius= 3.35, y radius= 3.35]   ;
			%Straight Lines [id:da4700528858249797] 
			\draw    (489.55,145.95) -- (545.9,145.75) ;
			\draw [shift={(545.9,145.75)}, rotate = 359.8] [color={rgb, 255:red, 0; green, 0; blue, 0 }  ][fill={rgb, 255:red, 0; green, 0; blue, 0 }  ][line width=0.75]      (0, 0) circle [x radius= 3.35, y radius= 3.35]   ;
			\draw [shift={(489.55,145.95)}, rotate = 359.8] [color={rgb, 255:red, 0; green, 0; blue, 0 }  ][fill={rgb, 255:red, 0; green, 0; blue, 0 }  ][line width=0.75]      (0, 0) circle [x radius= 3.35, y radius= 3.35]   ;
			%Straight Lines [id:da7102180211016452] 
			\draw    (563.12,92.1) -- (555.11,117.05) -- (545.9,145.75) ;
			\draw [shift={(545.9,145.75)}, rotate = 107.8] [color={rgb, 255:red, 0; green, 0; blue, 0 }  ][fill={rgb, 255:red, 0; green, 0; blue, 0 }  ][line width=0.75]      (0, 0) circle [x radius= 3.35, y radius= 3.35]   ;
			\draw [shift={(563.12,92.1)}, rotate = 107.8] [color={rgb, 255:red, 0; green, 0; blue, 0 }  ][fill={rgb, 255:red, 0; green, 0; blue, 0 }  ][line width=0.75]      (0, 0) circle [x radius= 3.35, y radius= 3.35]   ;
			%Straight Lines [id:da8065989702546641] 
			\draw    (563.12,92.1) -- (517.42,59.14) ;
			\draw [shift={(517.42,59.14)}, rotate = 215.8] [color={rgb, 255:red, 0; green, 0; blue, 0 }  ][fill={rgb, 255:red, 0; green, 0; blue, 0 }  ][line width=0.75]      (0, 0) circle [x radius= 3.35, y radius= 3.35]   ;
			\draw [shift={(563.12,92.1)}, rotate = 215.8] [color={rgb, 255:red, 0; green, 0; blue, 0 }  ][fill={rgb, 255:red, 0; green, 0; blue, 0 }  ][line width=0.75]      (0, 0) circle [x radius= 3.35, y radius= 3.35]   ;
			%Straight Lines [id:da9990328675708291] 
			\draw    (471.95,92.42) -- (517.42,59.14) ;
			\draw [shift={(517.42,59.14)}, rotate = 323.8] [color={rgb, 255:red, 0; green, 0; blue, 0 }  ][fill={rgb, 255:red, 0; green, 0; blue, 0 }  ][line width=0.75]      (0, 0) circle [x radius= 3.35, y radius= 3.35]   ;
			\draw [shift={(471.95,92.42)}, rotate = 323.8] [color={rgb, 255:red, 0; green, 0; blue, 0 }  ][fill={rgb, 255:red, 0; green, 0; blue, 0 }  ][line width=0.75]      (0, 0) circle [x radius= 3.35, y radius= 3.35]   ;
			%Straight Lines [id:da5562456405654405] 
			\draw    (489.55,145.95) ;
			\draw [shift={(489.55,145.95)}, rotate = 0] [color={rgb, 255:red, 0; green, 0; blue, 0 }  ][fill={rgb, 255:red, 0; green, 0; blue, 0 }  ][line width=0.75]      (0, 0) circle [x radius= 3.35, y radius= 3.35]   ;
			\draw [shift={(489.55,145.95)}, rotate = 0] [color={rgb, 255:red, 0; green, 0; blue, 0 }  ][fill={rgb, 255:red, 0; green, 0; blue, 0 }  ][line width=0.75]      (0, 0) circle [x radius= 3.35, y radius= 3.35]   ;
			%Straight Lines [id:da470575210336032] 
			\draw    (471.95,92.42) ;
			\draw [shift={(471.95,92.42)}, rotate = 0] [color={rgb, 255:red, 0; green, 0; blue, 0 }  ][fill={rgb, 255:red, 0; green, 0; blue, 0 }  ][line width=0.75]      (0, 0) circle [x radius= 3.35, y radius= 3.35]   ;
			\draw [shift={(471.95,92.42)}, rotate = 0] [color={rgb, 255:red, 0; green, 0; blue, 0 }  ][fill={rgb, 255:red, 0; green, 0; blue, 0 }  ][line width=0.75]      (0, 0) circle [x radius= 3.35, y radius= 3.35]   ;
			%Straight Lines [id:da1730651546898233] 
			\draw    (517.42,59.14) ;
			\draw [shift={(517.42,59.14)}, rotate = 0] [color={rgb, 255:red, 0; green, 0; blue, 0 }  ][fill={rgb, 255:red, 0; green, 0; blue, 0 }  ][line width=0.75]      (0, 0) circle [x radius= 3.35, y radius= 3.35]   ;
			\draw [shift={(517.42,59.14)}, rotate = 0] [color={rgb, 255:red, 0; green, 0; blue, 0 }  ][fill={rgb, 255:red, 0; green, 0; blue, 0 }  ][line width=0.75]      (0, 0) circle [x radius= 3.35, y radius= 3.35]   ;
			%Straight Lines [id:da548125248169927] 
			\draw    (545.9,145.75) ;
			\draw [shift={(545.9,145.75)}, rotate = 0] [color={rgb, 255:red, 0; green, 0; blue, 0 }  ][fill={rgb, 255:red, 0; green, 0; blue, 0 }  ][line width=0.75]      (0, 0) circle [x radius= 3.35, y radius= 3.35]   ;
			\draw [shift={(545.9,145.75)}, rotate = 0] [color={rgb, 255:red, 0; green, 0; blue, 0 }  ][fill={rgb, 255:red, 0; green, 0; blue, 0 }  ][line width=0.75]      (0, 0) circle [x radius= 3.35, y radius= 3.35]   ;
			%Straight Lines [id:da5206016168558385] 
			\draw    (563.12,92.1) ;
			\draw [shift={(563.12,92.1)}, rotate = 0] [color={rgb, 255:red, 0; green, 0; blue, 0 }  ][fill={rgb, 255:red, 0; green, 0; blue, 0 }  ][line width=0.75]      (0, 0) circle [x radius= 3.35, y radius= 3.35]   ;
			\draw [shift={(563.12,92.1)}, rotate = 0] [color={rgb, 255:red, 0; green, 0; blue, 0 }  ][fill={rgb, 255:red, 0; green, 0; blue, 0 }  ][line width=0.75]      (0, 0) circle [x radius= 3.35, y radius= 3.35]   ;
			%Straight Lines [id:da8775717125820396] 
			\draw    (517.42,59.14) -- (523.43,18.96) ;
			\draw [shift={(523.43,18.96)}, rotate = 278.51] [color={rgb, 255:red, 0; green, 0; blue, 0 }  ][fill={rgb, 255:red, 0; green, 0; blue, 0 }  ][line width=0.75]      (0, 0) circle [x radius= 3.35, y radius= 3.35]   ;
			\draw [shift={(517.42,59.14)}, rotate = 278.51] [color={rgb, 255:red, 0; green, 0; blue, 0 }  ][fill={rgb, 255:red, 0; green, 0; blue, 0 }  ][line width=0.75]      (0, 0) circle [x radius= 3.35, y radius= 3.35]   ;
			%Straight Lines [id:da776040593671999] 
			\draw    (473.43,39.96) -- (523.43,18.96) ;
			\draw [shift={(523.43,18.96)}, rotate = 337.22] [color={rgb, 255:red, 0; green, 0; blue, 0 }  ][fill={rgb, 255:red, 0; green, 0; blue, 0 }  ][line width=0.75]      (0, 0) circle [x radius= 3.35, y radius= 3.35]   ;
			\draw [shift={(473.43,39.96)}, rotate = 337.22] [color={rgb, 255:red, 0; green, 0; blue, 0 }  ][fill={rgb, 255:red, 0; green, 0; blue, 0 }  ][line width=0.75]      (0, 0) circle [x radius= 3.35, y radius= 3.35]   ;
			%Straight Lines [id:da831030969485757] 
			\draw    (473.43,39.96) -- (517.42,59.14) ;
			\draw [shift={(517.42,59.14)}, rotate = 23.55] [color={rgb, 255:red, 0; green, 0; blue, 0 }  ][fill={rgb, 255:red, 0; green, 0; blue, 0 }  ][line width=0.75]      (0, 0) circle [x radius= 3.35, y radius= 3.35]   ;
			\draw [shift={(473.43,39.96)}, rotate = 23.55] [color={rgb, 255:red, 0; green, 0; blue, 0 }  ][fill={rgb, 255:red, 0; green, 0; blue, 0 }  ][line width=0.75]      (0, 0) circle [x radius= 3.35, y radius= 3.35]   ;
			
			% Text Node
			\draw (172,152) node [anchor=north west][inner sep=0.75pt]   [align=left] {one-odd cycle graph};
			% Text Node
			\draw (328,152) node [anchor=north west][inner sep=0.75pt]   [align=left] {fused-odd graph};
			% Text Node
			\draw (462,153) node [anchor=north west][inner sep=0.75pt]   [align=left] {fused-odd graph};

		\end{tikzpicture}

	\end{center}

\caption{Examples of one-odd cycle and fused-odd graphs.}

	\label{asidj123llllk}\label{asopdkakso123123}
	
\end{figure}
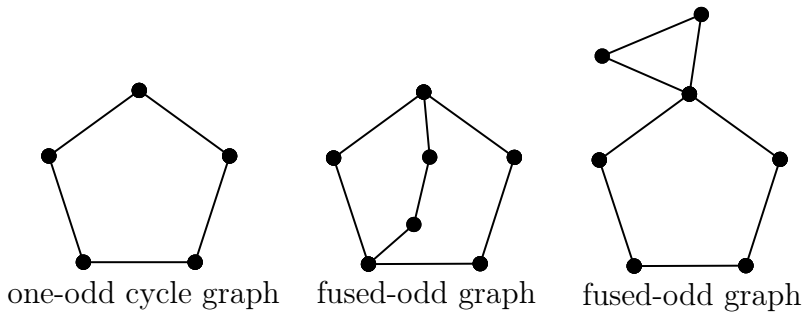

\noindent Suppose now that $G_{1}$ is obtained from $G_{0}$ by adding a
pendant.
Let $k$ denote the length of the path $P$ that connects the two odd cycles
in $G_{1}$.
\begin{itemize}
	\item If $k$ is odd, then we say that $G$ is an \emph{odd-linked} graph;
	see \cref{asidj123llllkasidj123llllk}.
	\item If $k$ is even, then we say that $G$ is an \emph{even-linked} graph;
	see \cref{asidj123llllkasidj123llllk}.
\end{itemize}

\noindent In the case of an odd(even)-linked graph, note that in the
ear-pendant decomposition of $G$ no further pendants are added after
$G_{1}$, since this would generate an additional odd cycle.
Therefore, for $i=2,\dots,p$, the graph $G_{i}$ is obtained from $G_{i-1}$
by adding an odd ear.
Moreover, both end-points of each added odd ear are contained in the path
$P$; otherwise, it is easy to see that a new odd cycle is created (see
\cref{asidj123llllkasidj123llllkasidj123llllk}).

\begin{figure}[H]
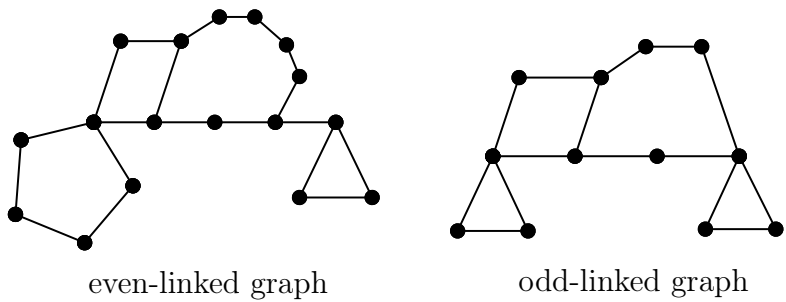

	
	\begin{center}

		\tikzset{every picture/.style={line width=0.75pt}} %set default line width to 0.75pt        
		


	\end{center}

\caption{Example of even-linked and odd-linked graphs.}
`

	\label{asidj123llllkasidj123llllk}
	
\end{figure}

\begin{figure}[H]
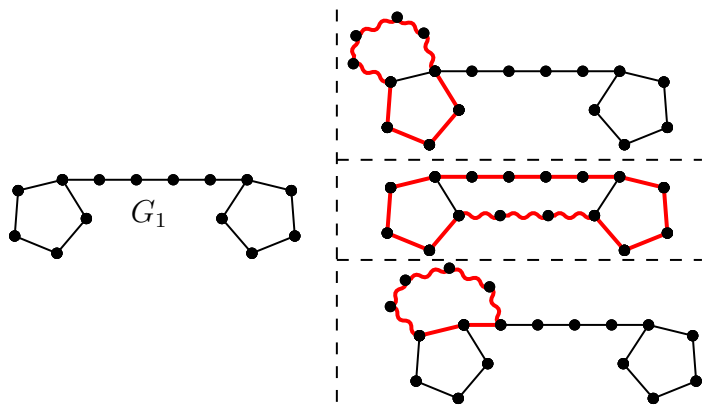

	
	\begin{center}

	\tikzset{every picture/.style={line width=0.75pt}} %set default line width to 0.75pt        
	
	%

	\end{center}

\caption{Improper addition of an odd ear produces additional odd cycles:
	three representative configurations.}

	\label{asidj123llllkasidj123llllkasidj123llllk}
	
\end{figure}

The structure of one-odd cycle and fused-odd graphs is relatively simple.
By contrast, odd-linked and even-linked graphs exhibit a strong analogy with
classical bipartite matching-covered graphs.

These four families exhaust all $2$-bicritical graphs with at most two odd
cycles and correspond to fundamentally different behaviors of maximum
independent sets and matchings.

The study of bipartite matching covered graphs has a long history: K\H{o}nig already
used this concept in 1915 while analyzing determinant decompositions
\cite{konig1915}, and nearly half a century later Hetyei formalized the term
\emph{elementary} instead bipartite matching covered and developed a classical characterization \cite{hetyei1964}.
It should be noted that several results concerning matching covered graphs,
such as those in \cite{LP}, are presented in the more general context of bipartite
matching covered graphs.

\medskip

In this paper, however, we will work with the following equivalent notion,
which is more convenient for our purposes. A graph $G$ is called a \emph{bipartite matching covered graph} if $G$ is
connected, bipartite, and every edge of $G$ belongs to a perfect matching.

\medskip

Bipartite matching covered graphs admit a particularly elegant structural
description in terms of ear decompositions, due to Lovász.

\begin{theorem}[\cite{lovasz1983ear}\label{oi12j3j123ijj}]
	Given any bipartite matching covered graph $G$,
	there exist a odd length ear decomposition
	\[
	G_{1},G_{2},\dots,G_{r}=G
	\]
	\noindent of matching covered subgraphs of $G$ where $G_{1}=K_{2}$.
\end{theorem}

In the bipartite setting, odd ear decompositions starting from $K_2$ characterize matching covered graphs.

\medskip
\medskip

More precisely, let $G$ be an odd(even)-linked graph with an ear-pendant
decomposition
\[
G_{0},\dots,G_{p}=G,
\]
\noindent where $G_{0}$ is an odd cycle and $G_{1}$ is obtained
from $G_{0}$ by adding a pendant. Let $P$ be the path that connects
the two cycles of $G_{1}$. Let $G^{\P}$ be the graph obtained from
$G$ by deleting the vertices of the cycles of $G_{1}$ that do not belong
to $V(P)$. Then the following result holds.

$ $

\begin{proposition}\label{proposition}
	If $G$ is an odd(even)-linked graph, then there exists a bipartite
	matching-covered graph $H$ such that $G^{\P}$ is obtained from $H$
	by deleting two (respectively, one) vertices.
\end{proposition}

\begin{proof}
	The proof is essentially an application of Lovász's ear-decomposition
	of bipartite matching-covered graphs. Suppose that $G$ is an odd-linked
	graph; the even case is similar.
	
	\medskip
	
	Note that $G^{\P}$ is bipartite and is obtained from $P$ by adding a
	sequence of $p-1$ odd ears.
	
	\[
	P,P_{2},\dots,P_{p}=G^{\P},
	\]
	
	\noindent where $P_{i}$ is obtained from $P_{i-1}$ by adding an odd ear,
	for each $i=2,\dots,p$. Note that this is achieved by choosing
	$P_{i}=G_{i}-(V(G)-V(G^{\P}))$.
	
	Add a fictitious vertex $x\notin V(G^{\P})$, and let $y$ and $z$ be the
	end-points of $P$. We define
	$P_{0}^{\P}:=(\{x,y\},\{xy\})\iso K_{2}$ and
	$P^{\P}=\left(P_{0}^{\P}\u P\right)+xz$. Then,
	
	\[
	P^{\P}_{0},P^{\P},P^{\P}\u P_{2},\dots,P^{\P}\u P_{p}=P^{\P}\u G^{\P},
	\]
	
	\noindent is an odd-ear decomposition starting from $K_{2}$. Hence,
	by \cref{oi12j3j123ijj}, the graph
	\[
	H:=P^{\P}\u G^{\P}
	\]
	is a bipartite matching-covered graph, and it satisfies $H-x=G^{\P}$;
	see \cref{asiodjh12} for a simple illustration of this proof.
	
	The even case is proved using a similar argument, but adding two
	fictitious vertices instead of one.
\end{proof}

\begin{figure}[H]
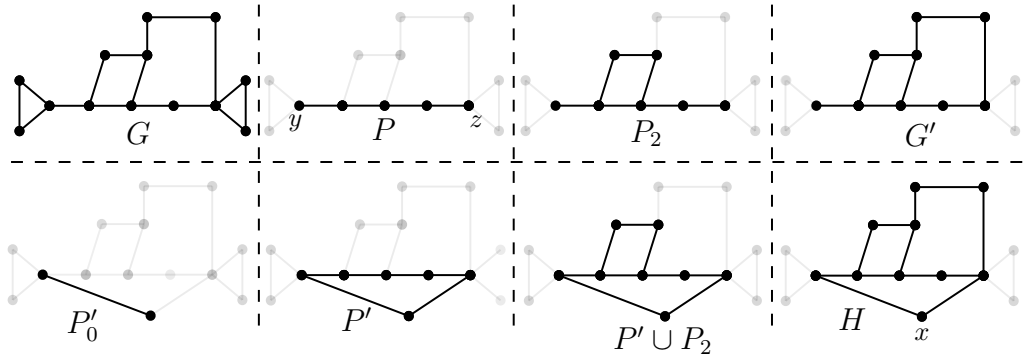

	
	\begin{center}

			\tikzset{every picture/.style={line width=0.75pt}} %set default line width to 0.75pt        
		


	\end{center}

\caption{An illustration of the proof of \cref{oi12j3j123ijj}}

	\label{asiodjh12}
	
\end{figure}

\medskip
\medskip

In the following sections, we use this classification to explicitly compute
$\core G$, $\corona G$, and $\alpha(G)$ for each family, and to study the
structural relationships among them.

\section{Bicyclic even-linked graphs}\label{sss3}

In this section, we determine explicitly $\alpha(G)$, $\core G$, and
$\corona G$ for $2$-bicritical even-linked graphs. We begin by establishing some auxiliary results needed for the main proofs.

\begin{theorem}
	[\label{19}\cite{berge2005some}] An independent set
	$S$ is maximum if and only if every independent set disjoint
	from $S$ can be matched into $S$. 
\end{theorem}

\begin{theorem}\label{19k}
	Let $G$ be a graph and let $S$ be an independent set. For a vertex $v\in S$,
	we have that $v\in \core G$ and $S$ is a maximum independent set if and only if
	every independent set $T$ disjoint from $S$ can be matched into
	$S\setminus\{v\}$.
\end{theorem}

\begin{proof}
	$(\Rightarrow)$
	Suppose that $v\in \core G$ and that $S$ is a maximum independent set of $G$.
	Since $v\in \core G$, we have
	\[
	\alpha(G-v)=\alpha(G)-1.
	\]
	In particular, $S\setminus\{v\}$ is a maximum independent set of $G-v$.
	
	Now let $T$ be an independent set disjoint from $S$. Then $T$ is also
	an independent set of $G-v$ and is disjoint from $S\setminus\{v\}$.
	By \cref{19}, applied to the graph $G-v$, it follows that
	$T$ can be matched into $S\setminus\{v\}$.
	
	$(\Leftarrow)$
	Suppose now that every independent set $T$ disjoint from $S$
	can be matched into $S\setminus\{v\}$.
	Then, by \cref{19} applied to the graph $G-v$,
	it follows that $S\setminus\{v\}$ is a maximum independent set of $G-v$.
	Consequently,
	\[
	\alpha(G-v)=|S|-1,
	\]
	and therefore $S$ is a maximum independent set of $G$.
	
	Moreover, since the deletion of $v$ reduces the independence number by one,
	it follows that $v$ belongs to every maximum independent set of $G$, that is,
	$v\in \core G$.
\end{proof}

\begin{theorem}[\cite{hetyei1964}\label{hetyei}]
	Let $G$ be a bipartite graph with bipartition $(A,B)$. Then the following 
	statements are equivalent:
	\begin{enumerate}
		\item $G$ is elementary (that is, the union of perfect matchings in it induces a connected
		subgraph).
		\item $G$ has exactly two minimum vertex covers, namely $A$ and $B$.
		\item $\lvert A\rvert = \lvert B\rvert$ and 
		$\lvert S\rvert + 1 \le \lvert N(S)\rvert$ 
		for every non-empty proper subset $S \subset A$.
		\item Either $G = K_{2}$, or $\lvert V(G)\rvert \ge 4$ and for every $v \in A$ and $w \in B$, the graph $G - v - w$ has a perfect matching.
		\item $G$ is matching covered graph (that is $G$ is connected and every edge of $G$ is contained in some perfect matching).
	\end{enumerate}
\end{theorem}

\begin{theorem}\label{12oij3}
	Let $G$ be an even-linked graph of order $n$, let $C$ and $C^{\P}$ be the
	unique odd cycles of $G$, and let $X$ be the set of vertices in
	$V(C)\u V(C^{\P})$ having degree $2$ in $G$. Then $G-X$ is a graph with
	bipartition $(A,B)$ such that $|A|=|B|+1$. Moreover,
	\begin{itemize}
		\item $\core G=B$, 
		\item $\alpha(G)=\frac{n-1}{2}$,
		\item $\a{\core G}=\frac{n-\a{V(C)\u V(C^{\P})}+1}{2}$.
	\end{itemize}
\end{theorem}

\begin{proof}
	Let $G_{0},G_{1},\dots,G_{p}=G$ be an ear-pendant decomposition of $G$.
	Then $G_{0}$ is an odd cycle and $G_{1}$ is obtained from $G_{0}$
	by adding a pendant. Moreover, for every $i=2,\dots,p$,
	the graph $G_{i}$ is obtained from $G_{i-1}$ by adding an odd ear.
	Observe that none of these ears uses vertices from $X$.
	Consequently, the graph
	\[
	G^{\P}:=G-X
	\]
	is bipartite and connected; let $(A,B)$ be a bipartition of $G^{\P}$.
	
	Arguing as in the proof of \cref{proposition}, the graph $G^{\P}$
	admits the following decomposition:
	\[
	P=P_{1},P_{2},\dots,P_{p}=G^{\P},
	\]
	where each $P_{i}$ is obtained from $P_{i-1}$ by adding an odd ear,
	and $P$ is the path that connects both cycles in $G_{1}$.
	Since $P$ has even length, adding each odd ear introduces
	an even number of new vertices. Therefore, $G-X$ has odd order
	and, without loss of generality, we may assume that
	\[
	|A|=|B|+1.
	\]
	
	\medskip
	
	Let $x$ and $y$ be the end-points of $P$; note that $x,y\in A$.
	Let $S^{\P}$ be a maximum independent set of $C\cup C^{\P}$ that
	contains neither $x$ nor $y$. We define
	\[
	S:=S^{\P}\cup B.
	\]
	
	We show that $S$ is a maximum independent set of $G$
	using \cref{19k}. Indeed, fixing $v\in B$ and given any
	independent set $T$ disjoint from $S$, we prove that $T$ can
	be matched into $S\setminus\{v\}$. In this way, it follows that
	$S$ is a maximum independent set of $G$ and that
	$B\subseteq \core G.$
	
	\begin{claim}
		The set $T$ can be matched into $S\setminus\{v\}$.
	\end{claim}
	
	\begin{proof}
		For an illustration of the proof, see \cref{1231231dd}.
		Let $H$ be the graph obtained from $G$ by adding a new vertex $w$ such that
		\[
		N_{H}(w)=\{x,y\}.
		\]
		Arguing as in the proof of \cref{proposition}, it follows that $H-X$
		is a bipartite matching-covered graph. In particular, by \cref{hetyei},
		the graph $(H-X)-v-x$ has a perfect matching $M$.
		
		Observe that necessarily $M(w)=y$, since $w$ has a unique neighbor in
		$(H-X)-v-x$. We define
		\[
		M^{\P}:=M-\{xy\}.
		\]
		Then $M^{\P}$ is a matching of $G-X$ that leaves exactly the vertices
		$x,y$, and $v$ unsaturated.
		
		On the other hand, the set $S\cap V(C)$ is a maximum independent set
		of $C$. By \cref{19}, the set $T\cap V(C)$ can be matched into
		$S\cap V(C)$ by means of a matching $M_{1}$. Similarly,
		$T\cap V(C^{\P})$ can be matched into $S\cap V(C^{\P})$ by means
		of a matching $M_{2}$. In particular,
		\[
		M_{1}\subseteq E(C)
		\quad\text{and}\quad
		M_{2}\subseteq E(C^{\P}).
		\]
		
		Finally, the vertices of
		$T\setminus\bigl(V(C)\cup V(C^{\P})\bigr)$ can be matched into
		$S\setminus\{v\}$ using the matching $M^{\P}$.
		Therefore,
		\[
		M_{1}\cup M_{2}\cup M^{\P}
		\]
		is a matching that saturates all vertices of $T$ in $S\setminus\{v\}$,
		which completes the proof.
	\end{proof}

\begin{comment}
	\textbf{Claim.} El conjunto $T$ se puede matchear en $S\setminus\{v\}$. 

\emph{proof.} Para acompañar la demostración ver la FiguraXXX. Sea
$H$ el grafo obtenido de $G$ al agregar un nuevo vértice $w$ tal
que $N_{H}(w)=\{x,y\}$. Razonando como en la prueba del TeoremaXXX
$H-X$ es un bipartito matching covered. Entonces por el TeoremaXXXhetyei
$\left(H-X\right)-v-x$ tiene matching perfecto $M$. Notar que $M(w)=y$,
pues $w$ tiene un único vecino en $\left(H-X\right)-v-x$ . Si $M^{\P}:=M-\{x,y\}$
notar que $M^{\P}$ es un matching de $G-X$ que deja sin saturar
únicamente a $x,y$ and $v$. 

Observar que $S\ii V(C)$ es un independiente máximo de $S$, por
el TeoremaXXXberge $T\ii V(C)$ se puede matchear con un matching
digamos $M_{1}$ en $V(C)$, de la misma forma $T\ii V(C^{\P})$ se
puede matchear en $S\ii V(C^{\P})$ por un matching que llamaremos
$M_{2}$. Entonces claramente $M_{1}\s E(C)$ y $M_{2}\s E(C^{\P})$.
Por otro lado, los vértices de $A-\{x,y\}$ se pueden matchear en
$S$ desde el matching $M^{\P}$. Por último el matching

\[
M_{1}\u M_{2}\u M^{\P}
\]

\noindent es un matching de $T$ en $S-\{v\}$.
\end{comment}

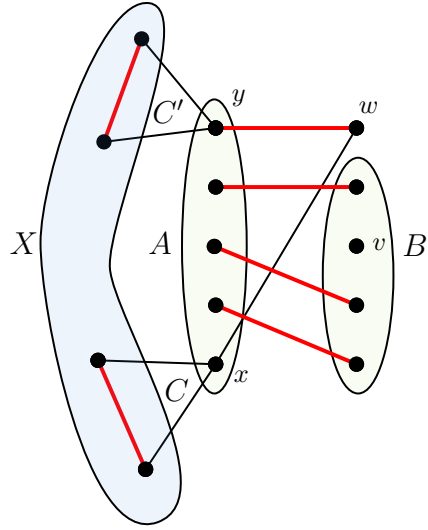
\begin{figure}[H]
	
	\begin{center}

		\tikzset{every picture/.style={line width=0.75pt}} %set default line width to 0.75pt        
		
		\begin{tikzpicture}[x=0.75pt,y=0.75pt,yscale=-1,xscale=1]
			%uncomment if require: \path (0,421); %set diagram left start at 0, and has height of 421
			
			%Shape: Ellipse [id:dp39680171537178943] 
			\draw  [fill={rgb, 255:red, 184; green, 233; blue, 134 }  ,fill opacity=0.1 ] (243.36,177.07) .. controls (243.36,136.08) and (250.66,102.85) .. (259.68,102.85) .. controls (268.69,102.85) and (276,136.08) .. (276,177.07) .. controls (276,218.06) and (268.69,251.29) .. (259.68,251.29) .. controls (250.66,251.29) and (243.36,218.06) .. (243.36,177.07) -- cycle ;
			%Straight Lines [id:da9658948306259099] 
			\draw [color={rgb, 255:red, 255; green, 0; blue, 0 }  ,draw opacity=1 ][line width=1.5]    (223,72.29) -- (204,124.29) ;
			%Straight Lines [id:da8085422222310005] 
			\draw    (260.36,236.52) -- (331.36,117.29) ;
			\draw [shift={(331.36,117.29)}, rotate = 300.77] [color={rgb, 255:red, 0; green, 0; blue, 0 }  ][fill={rgb, 255:red, 0; green, 0; blue, 0 }  ][line width=0.75]      (0, 0) circle [x radius= 3.35, y radius= 3.35]   ;
			\draw [shift={(260.36,236.52)}, rotate = 300.77] [color={rgb, 255:red, 0; green, 0; blue, 0 }  ][fill={rgb, 255:red, 0; green, 0; blue, 0 }  ][line width=0.75]      (0, 0) circle [x radius= 3.35, y radius= 3.35]   ;
			%Shape: Ellipse [id:dp4969754068257173] 
			\draw  [fill={rgb, 255:red, 184; green, 233; blue, 134 }  ,fill opacity=0.1 ] (314.36,191.79) .. controls (314.36,158.92) and (322.33,132.29) .. (332.18,132.29) .. controls (342.02,132.29) and (350,158.92) .. (350,191.79) .. controls (350,224.65) and (342.02,251.29) .. (332.18,251.29) .. controls (322.33,251.29) and (314.36,224.65) .. (314.36,191.79) -- cycle ;
			%Straight Lines [id:da39708895429607993] 
			\draw    (260.36,236.52) -- (225,289.57) ;
			\draw [shift={(225,289.57)}, rotate = 123.68] [color={rgb, 255:red, 0; green, 0; blue, 0 }  ][fill={rgb, 255:red, 0; green, 0; blue, 0 }  ][line width=0.75]      (0, 0) circle [x radius= 3.35, y radius= 3.35]   ;
			\draw [shift={(260.36,236.52)}, rotate = 123.68] [color={rgb, 255:red, 0; green, 0; blue, 0 }  ][fill={rgb, 255:red, 0; green, 0; blue, 0 }  ][line width=0.75]      (0, 0) circle [x radius= 3.35, y radius= 3.35]   ;
			%Straight Lines [id:da04898079123590626] 
			\draw    (260.36,236.52) -- (201,234.57) ;
			\draw [shift={(201,234.57)}, rotate = 181.88] [color={rgb, 255:red, 0; green, 0; blue, 0 }  ][fill={rgb, 255:red, 0; green, 0; blue, 0 }  ][line width=0.75]      (0, 0) circle [x radius= 3.35, y radius= 3.35]   ;
			\draw [shift={(260.36,236.52)}, rotate = 181.88] [color={rgb, 255:red, 0; green, 0; blue, 0 }  ][fill={rgb, 255:red, 0; green, 0; blue, 0 }  ][line width=0.75]      (0, 0) circle [x radius= 3.35, y radius= 3.35]   ;
			%Straight Lines [id:da13193524738241713] 
			\draw    (260.36,117.29) -- (223,72.29) ;
			\draw [shift={(223,72.29)}, rotate = 230.3] [color={rgb, 255:red, 0; green, 0; blue, 0 }  ][fill={rgb, 255:red, 0; green, 0; blue, 0 }  ][line width=0.75]      (0, 0) circle [x radius= 3.35, y radius= 3.35]   ;
			\draw [shift={(260.36,117.29)}, rotate = 230.3] [color={rgb, 255:red, 0; green, 0; blue, 0 }  ][fill={rgb, 255:red, 0; green, 0; blue, 0 }  ][line width=0.75]      (0, 0) circle [x radius= 3.35, y radius= 3.35]   ;
			%Straight Lines [id:da6191388104973681] 
			\draw    (204,124.29) -- (260.36,117.29) ;
			\draw [shift={(260.36,117.29)}, rotate = 352.92] [color={rgb, 255:red, 0; green, 0; blue, 0 }  ][fill={rgb, 255:red, 0; green, 0; blue, 0 }  ][line width=0.75]      (0, 0) circle [x radius= 3.35, y radius= 3.35]   ;
			\draw [shift={(204,124.29)}, rotate = 352.92] [color={rgb, 255:red, 0; green, 0; blue, 0 }  ][fill={rgb, 255:red, 0; green, 0; blue, 0 }  ][line width=0.75]      (0, 0) circle [x radius= 3.35, y radius= 3.35]   ;
			%Straight Lines [id:da4059810469200744] 
			\draw [color={rgb, 255:red, 255; green, 0; blue, 0 }  ,draw opacity=1 ][line width=1.5]    (201,234.57) -- (225,289.57) ;
			%Straight Lines [id:da04882836287068559] 
			\draw [color={rgb, 255:red, 255; green, 0; blue, 0 }  ,draw opacity=1 ][line width=1.5]    (260.36,147.09) -- (331.36,147.09) ;
			%Straight Lines [id:da7973581653535847] 
			\draw [color={rgb, 255:red, 255; green, 0; blue, 0 }  ,draw opacity=1 ][line width=1.5]    (259.68,177.07) -- (331.36,206.71) ;
			%Straight Lines [id:da9975380399871326] 
			\draw [color={rgb, 255:red, 255; green, 0; blue, 0 }  ,draw opacity=1 ][line width=1.5]    (260.36,206.71) -- (331.36,236.52) ;
			%Straight Lines [id:da5895895028106892] 
			\draw [color={rgb, 255:red, 255; green, 0; blue, 0 }  ,draw opacity=1 ][line width=1.5]    (260.36,117.29) -- (331.36,117.29) ;
			%Shape: Polygon Curved [id:ds8848626437164268] 
			\draw  [fill={rgb, 255:red, 74; green, 144; blue, 226 }  ,fill opacity=0.1 ] (227,54.29) .. controls (252,58.14) and (205,163.14) .. (207,188.14) .. controls (209,213.14) and (267,303.29) .. (231,316.29) .. controls (195,329.29) and (172,208.29) .. (172,175.29) .. controls (172,142.29) and (202,50.43) .. (227,54.29) -- cycle ;
			%Straight Lines [id:da133861302897053] 
			\draw    (331.36,176.9) ;
			\draw [shift={(331.36,176.9)}, rotate = 0] [color={rgb, 255:red, 0; green, 0; blue, 0 }  ][fill={rgb, 255:red, 0; green, 0; blue, 0 }  ][line width=0.75]      (0, 0) circle [x radius= 3.35, y radius= 3.35]   ;
			\draw [shift={(331.36,176.9)}, rotate = 0] [color={rgb, 255:red, 0; green, 0; blue, 0 }  ][fill={rgb, 255:red, 0; green, 0; blue, 0 }  ][line width=0.75]      (0, 0) circle [x radius= 3.35, y radius= 3.35]   ;
			%Straight Lines [id:da8179593943783703] 
			\draw    (331.36,147.09) ;
			\draw [shift={(331.36,147.09)}, rotate = 0] [color={rgb, 255:red, 0; green, 0; blue, 0 }  ][fill={rgb, 255:red, 0; green, 0; blue, 0 }  ][line width=0.75]      (0, 0) circle [x radius= 3.35, y radius= 3.35]   ;
			\draw [shift={(331.36,147.09)}, rotate = 0] [color={rgb, 255:red, 0; green, 0; blue, 0 }  ][fill={rgb, 255:red, 0; green, 0; blue, 0 }  ][line width=0.75]      (0, 0) circle [x radius= 3.35, y radius= 3.35]   ;
			%Straight Lines [id:da10643365247213787] 
			\draw    (331.36,117.29) ;
			\draw [shift={(331.36,117.29)}, rotate = 0] [color={rgb, 255:red, 0; green, 0; blue, 0 }  ][fill={rgb, 255:red, 0; green, 0; blue, 0 }  ][line width=0.75]      (0, 0) circle [x radius= 3.35, y radius= 3.35]   ;
			\draw [shift={(331.36,117.29)}, rotate = 0] [color={rgb, 255:red, 0; green, 0; blue, 0 }  ][fill={rgb, 255:red, 0; green, 0; blue, 0 }  ][line width=0.75]      (0, 0) circle [x radius= 3.35, y radius= 3.35]   ;
			%Straight Lines [id:da34483818806190736] 
			\draw    (331.36,236.52) ;
			\draw [shift={(331.36,236.52)}, rotate = 0] [color={rgb, 255:red, 0; green, 0; blue, 0 }  ][fill={rgb, 255:red, 0; green, 0; blue, 0 }  ][line width=0.75]      (0, 0) circle [x radius= 3.35, y radius= 3.35]   ;
			\draw [shift={(331.36,236.52)}, rotate = 0] [color={rgb, 255:red, 0; green, 0; blue, 0 }  ][fill={rgb, 255:red, 0; green, 0; blue, 0 }  ][line width=0.75]      (0, 0) circle [x radius= 3.35, y radius= 3.35]   ;
			%Straight Lines [id:da4102283484021404] 
			\draw    (331.36,206.71) ;
			\draw [shift={(331.36,206.71)}, rotate = 0] [color={rgb, 255:red, 0; green, 0; blue, 0 }  ][fill={rgb, 255:red, 0; green, 0; blue, 0 }  ][line width=0.75]      (0, 0) circle [x radius= 3.35, y radius= 3.35]   ;
			\draw [shift={(331.36,206.71)}, rotate = 0] [color={rgb, 255:red, 0; green, 0; blue, 0 }  ][fill={rgb, 255:red, 0; green, 0; blue, 0 }  ][line width=0.75]      (0, 0) circle [x radius= 3.35, y radius= 3.35]   ;
			%Straight Lines [id:da7890505115477866] 
			\draw    (260.36,206.71) ;
			\draw [shift={(260.36,206.71)}, rotate = 0] [color={rgb, 255:red, 0; green, 0; blue, 0 }  ][fill={rgb, 255:red, 0; green, 0; blue, 0 }  ][line width=0.75]      (0, 0) circle [x radius= 3.35, y radius= 3.35]   ;
			\draw [shift={(260.36,206.71)}, rotate = 0] [color={rgb, 255:red, 0; green, 0; blue, 0 }  ][fill={rgb, 255:red, 0; green, 0; blue, 0 }  ][line width=0.75]      (0, 0) circle [x radius= 3.35, y radius= 3.35]   ;
			%Straight Lines [id:da305554160528569] 
			\draw    (259.68,177.07) ;
			\draw [shift={(259.68,177.07)}, rotate = 0] [color={rgb, 255:red, 0; green, 0; blue, 0 }  ][fill={rgb, 255:red, 0; green, 0; blue, 0 }  ][line width=0.75]      (0, 0) circle [x radius= 3.35, y radius= 3.35]   ;
			\draw [shift={(259.68,177.07)}, rotate = 0] [color={rgb, 255:red, 0; green, 0; blue, 0 }  ][fill={rgb, 255:red, 0; green, 0; blue, 0 }  ][line width=0.75]      (0, 0) circle [x radius= 3.35, y radius= 3.35]   ;
			%Straight Lines [id:da6561700492050402] 
			\draw    (260.36,147.09) ;
			\draw [shift={(260.36,147.09)}, rotate = 0] [color={rgb, 255:red, 0; green, 0; blue, 0 }  ][fill={rgb, 255:red, 0; green, 0; blue, 0 }  ][line width=0.75]      (0, 0) circle [x radius= 3.35, y radius= 3.35]   ;
			\draw [shift={(260.36,147.09)}, rotate = 0] [color={rgb, 255:red, 0; green, 0; blue, 0 }  ][fill={rgb, 255:red, 0; green, 0; blue, 0 }  ][line width=0.75]      (0, 0) circle [x radius= 3.35, y radius= 3.35]   ;
			%Straight Lines [id:da02120580748697387] 
			\draw    (260.36,117.29) ;
			\draw [shift={(260.36,117.29)}, rotate = 0] [color={rgb, 255:red, 0; green, 0; blue, 0 }  ][fill={rgb, 255:red, 0; green, 0; blue, 0 }  ][line width=0.75]      (0, 0) circle [x radius= 3.35, y radius= 3.35]   ;
			\draw [shift={(260.36,117.29)}, rotate = 0] [color={rgb, 255:red, 0; green, 0; blue, 0 }  ][fill={rgb, 255:red, 0; green, 0; blue, 0 }  ][line width=0.75]      (0, 0) circle [x radius= 3.35, y radius= 3.35]   ;
			%Straight Lines [id:da4906763160337255] 
			\draw    (225,289.57) ;
			\draw [shift={(225,289.57)}, rotate = 0] [color={rgb, 255:red, 0; green, 0; blue, 0 }  ][fill={rgb, 255:red, 0; green, 0; blue, 0 }  ][line width=0.75]      (0, 0) circle [x radius= 3.35, y radius= 3.35]   ;
			\draw [shift={(225,289.57)}, rotate = 0] [color={rgb, 255:red, 0; green, 0; blue, 0 }  ][fill={rgb, 255:red, 0; green, 0; blue, 0 }  ][line width=0.75]      (0, 0) circle [x radius= 3.35, y radius= 3.35]   ;
			%Straight Lines [id:da09317791608718029] 
			\draw    (201,234.57) ;
			\draw [shift={(201,234.57)}, rotate = 0] [color={rgb, 255:red, 0; green, 0; blue, 0 }  ][fill={rgb, 255:red, 0; green, 0; blue, 0 }  ][line width=0.75]      (0, 0) circle [x radius= 3.35, y radius= 3.35]   ;
			\draw [shift={(201,234.57)}, rotate = 0] [color={rgb, 255:red, 0; green, 0; blue, 0 }  ][fill={rgb, 255:red, 0; green, 0; blue, 0 }  ][line width=0.75]      (0, 0) circle [x radius= 3.35, y radius= 3.35]   ;
			
			% Text Node
			\draw (338,171.4) node [anchor=north west][inner sep=0.75pt]  [font=\footnotesize]  {$v$};
			% Text Node
			\draw (268,238.4) node [anchor=north west][inner sep=0.75pt]  [font=\footnotesize]  {$x$};
			% Text Node
			\draw (267,96.4) node [anchor=north west][inner sep=0.75pt]  [font=\footnotesize]  {$y$};
			% Text Node
			\draw (331,102.4) node [anchor=north west][inner sep=0.75pt]  [font=\footnotesize]  {$w$};
			% Text Node
			\draw (227,101.4) node [anchor=north west][inner sep=0.75pt]    {$C^{\prime }$};
			% Text Node
			\draw (233,242.4) node [anchor=north west][inner sep=0.75pt]    {$C$};
			% Text Node
			\draw (155,168.4) node [anchor=north west][inner sep=0.75pt]    {$X$};
			% Text Node
			\draw (225,168.4) node [anchor=north west][inner sep=0.75pt]    {$A$};
			% Text Node
			\draw (353,169.4) node [anchor=north west][inner sep=0.75pt]    {$B$};

		\end{tikzpicture}

	\end{center}

\caption{An illustration of the proof of \cref{12oij3}}

	\label{1231231dd}
	
\end{figure}

Note that $S^{\p}$ on each cycle $C$ and $C^{\P}$ can be chosen in two
distinct disjoint ways, which proves that $\core G=B$.
On the other hand, note that

\begin{eqnarray*}
	2\alpha(G) & = & 2\a S\\
	& = & 2\a{S^{\P}}+2\a B\\
	& = & \left(\a C+\a{C^{\P}}-2\right)+\left(\a A+\a B-1\right)\\
	& = & \left(\a C+\a{C^{\P}}+\a A+\a B-2\right)-1\\
	& = & n-1,
\end{eqnarray*}

\noindent that is, $\alpha(G)=\frac{n-1}{2}$.

\medskip

Finally, note that

\begin{eqnarray*}
	2\a{\core G}+1 & = & 2\a B+1\\
	& = & \a B+\a A\\
	& = & n-\a{V(C)\u V(C^{\P})}+2.
\end{eqnarray*}

\noindent which implies that
$\a{\core G}=\frac{n-\a{V(C)\u V(C^{\P})}+1}{2}$.

\end{proof}

\begin{theorem}\label{12oij312oij3}
	With the notation and hypotheses of \cref{12oij3},
	we have
	\[
	\corona G=\left(V(C)\u V(C^{\P})\u B\right)-\{x,y\}=V(G)-A
	\]
	and, in particular,
	\[
	\a{\corona G}=\frac{n+\a{V(C)\u V(C^{\P})}-3}{2}.
	\]
\end{theorem}

\begin{proof}
	Note that $N(\core G)=N(B)=A$, since $G-X$ is a connected bipartite
	graph. That is, $\corona G\s V(G)-A$. On the other hand, as in the
	proof of \cref{12oij3}, the set $S^{\p}$ on each cycle $C$ and
	$C^{\P}$ can be chosen in two distinct disjoint ways, but neither
	$x$ nor $y$ belongs to $\corona G$. Therefore,
	$\left(V(C)\u V(C^{\P})\right)-\{x,y\}\s\corona G$.
	Hence,
	\[
	\corona G=\left(V(C)\u V(C^{\P})\u B\right)-\{x,y\}=V(G)-A.
	\]
	
	\noindent On the other hand, note that
	\begin{eqnarray*}
		2\a{\corona G} & = & 2\a{\left(V(C)\u V(C^{\P})\u B\right)-\{x,y\}}\\
		& = & 2\a{V(C)\u V(C^{\P})}+2\a B-4\\
		& = & 2\a{V(C)\u V(C^{\P})}+\a A+\a B-5\\
		& = & \left(\a A+\a B+\a{V(C)\u V(C^{\P})}-2\right)
		+\a{V(C)\u V(C^{\P})}-3\\
		& = & n+\a{V(C)\u V(C^{\P})}-3.
	\end{eqnarray*}
	
	\noindent As desired.
\end{proof}

\begin{theorem}\label{123kj123j123kj123j123kj123j123kj123j}
	Let $G$ be a bicyclic even-linked graph and let
	$A$ be the set defined in \cref{12oij3}. Then the following
	properties hold:
	\begin{itemize}
		\item $\a{\corona G}+\a{\core G}=2\alpha(G)$,
		\item $N(\core G)=A,$
		\item $\corona G$ and $N(\core G)$ form a partition of $V(G)$. 
	\end{itemize}
\end{theorem}

\begin{proof}
	By \cref{12oij3} and \cref{12oij312oij3}, we have
	\begin{eqnarray*}
		\a{\corona G}+\a{\core G}
		& = &
		\frac{n+\a{V(C)\u V(C^{\P})}-3}{2}
		+\frac{n-\a{V(C)\u V(C^{\P})}+1}{2}\\
		& = & \frac{2(n-1)}{2}\\
		& = & 2\alpha(G).
	\end{eqnarray*}
	
	\noindent In the proof of \cref{12oij312oij3} it is observed that
	$N(\core G)=N(B)=A$, and by \cref{12oij312oij3} we have
	$\corona G=V(G)-A$.
\end{proof}

\begin{theorem}\label{12oiu3h129u3}
	Let $G$ be an odd-linked graph. Then
	\[
	A(G)=\core G,\quad D(G)=V(G)-\core G
	\quad\text{and}\quad
	\mu(G)=\frac{n-1}{2}.
	\]
\end{theorem}

\begin{proof}
	With the notation of the proof of \cref{12oij3}, let $z\in A$.
	Then, since $H-X$ is a matching-covered graph, by \cref{hetyei}
	the graph $H-X-w-z$ has a perfect matching $M_{z}$.
	Let $M_{X}$ be a perfect matching of $G[X]$.
	Then $M_{z}\cup M_{X}$ is a maximum matching of $G$ that leaves
	only $z$ unsaturated. In particular, $A\subseteq D(G)$.
	Moreover, it follows that
	\[
	\mu(G)=\frac{n-1}{2}.
	\]
	
	\noindent Analogously, the matching $M_{x}\cup M_{X}$ leaves only
	$x$ unsaturated, and it is easy to modify the matching in $C$ in
	such a way that $V(C)\subseteq D(G)$; see \cref{asd1po2k3k123}.

\begin{figure}[H]
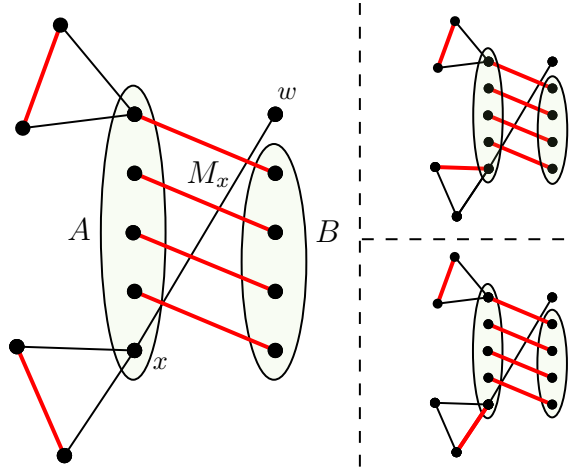

	
	\begin{center}

		\tikzset{every picture/.style={line width=0.75pt}} %set default line width to 0.75pt        
		


	\end{center}

\caption{Illustration of the proof of \cref{12oiu3h129u3}}

	\label{asd1po2k3k123}
	
\end{figure}

\noindent Analogously, by considering the matching $M_{y}\cup M_{X}$,
it follows that $V(C^{\P})\subseteq D(G)$.
On the other hand, if a vertex $b\in B$ belongs to $D(G)$, then,
since $A\subseteq D(G)$, such a vertex $b$ is contained in a
non-trivial component of $G[D(G)]$. But by \cref{ge}, this component
is a factor-critical graph, which cannot be bipartite, leading to the
conclusion that $b$ lies on an odd cycle, a contradiction. Therefore,
\[
D(G)=V(G)-B=V(G)-\core G
\quad\text{and}\quad
A(G)=\core G.
\]
\noindent As desired.
\end{proof}

\begin{corollary}\label{123kj123j123kj123j123kj123j}
Let $G$ be an even-linked graph of order
$n$. Then $\alpha(G)+\mu(G)=n-1$.
\end{corollary}

\section{Bicyclic odd-linked graphs}\label{sss4}

In this section, we study $2$-bicritical odd-linked graphs.
In contrast with the even-linked case, the behavior of $\core G$ and
$\corona G$ is particularly simple: we show that $\core G=\emptyset$ and
$\corona G=V(G)$, and we determine the corresponding matching structure.

\begin{theorem}\label{123oiu21j3jj}
	Let $G$ be a bicyclic odd-linked graph of
	order $n$. Then the following items hold.
	\begin{itemize}
		\item $G$ has a perfect matching,
		\item $\alpha(G)=\frac{n-2}{2}$,
		\item $\core G=\emptyset$, 
		\item $\corona G=V(G).$
	\end{itemize}
\end{theorem}
\begin{proof}
	For an illustration of the proof, see \cref{123123jsjsj}. Let $C$
	and $C^{\P}$ be the unique odd cycles of $G$, and let $X$ be the set
	of vertices in $V(C)\u V(C^{\P})$ having degree $2$ in $G$.
	Let $G_{0},G_{1},\dots,G_{p}=G$ be an ear-pendant decomposition of $G$.
	Then $G_{0}$ is an odd cycle and $G_{1}$ is obtained from $G_{0}$ by
	adding a pendant. Moreover, for every $i=2,\dots,p$, the graph
	$G_{i}$ is obtained from $G_{i-1}$ by adding an odd ear.
	Each of these ears does not use vertices from $X$.
	Thus $G^{\P}:=G-X$ is a connected bipartite graph, say with
	bipartition $(A,B)$. Then, as in the proof of \cref{proposition},
	$G^{\P}$ admits the following decomposition:
	\[
	P,P_{2},\dots,P_{p}=G^{\P}
	\]
	\noindent where each graph is obtained from the previous one by adding
	an odd ear, and $P$ is the path that connects both cycles in $G_{1}$.
	Since $P$ has odd length, after adding each odd ear we introduce
	an even number of new vertices; therefore $G-X$ has even order and
	$|A|=|B|$. A perfect matching in $P$ extends easily to a perfect
	matching in $P_{2}$, and proceeding in this way we obtain a perfect
	matching of $G^{\P}$, which extends—by choosing a near-perfect
	matching in each cycle $C$ and $C^{\P}$—to a perfect matching of $G$.
	
	But $G$ is a non-empty $2$-bicritical graph, and hence it is not a
	K\H{o}nig--Egerv\'ary graph, that is,
	\[
	\alpha(G)+\mu(G)=\alpha(G)+\frac{n}{2}<n,
	\]
	\noindent and therefore
	$\alpha(G)\le\frac{n}{2}-1=\frac{n-2}{2}$.

\begin{figure}[H]
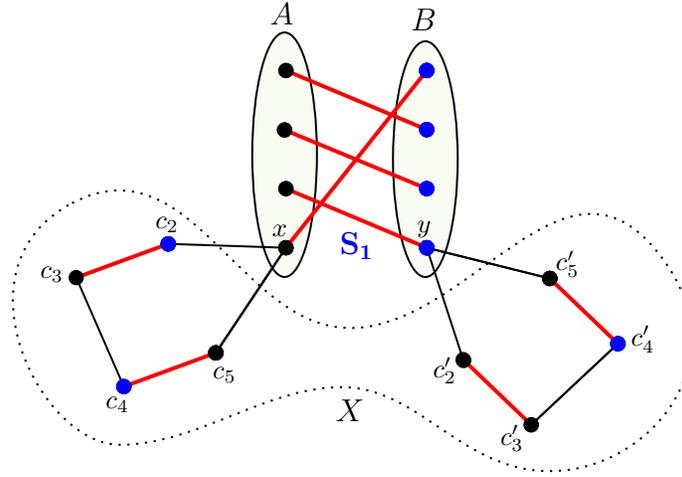

	
	\begin{center}

		\tikzset{every picture/.style={line width=0.75pt}} %set default line width to 0.75pt        
		


	\end{center}

\caption{Illustration of the proof of \cref{123oiu21j3jj}}

	\label{123123jsjsj}
	
\end{figure}

Let $x$ and $y$ be the end-points of $P$, and write $C=c_{1},\dots,c_{k}$
and $C^{\P}=c_{1}^{\P},\dots,c_{t}^{\P}$, where $c_{1}=c_{k}=x\in A$
and $c_{1}^{\P}=c_{t}^{\P}=y\in B$. Then the set

\begin{align*}
	S_{1} & =B\u\{c_{i}:i\text{ is even}\}\u\{c_{i}^{\P}:i\text{ is even and }i\neq2\},
\end{align*}

\noindent is an independent set of $G$ such that
\begin{eqnarray*}
	\a{S_{1}} & = & \a B+\frac{\a{V(C)}-1}{2}+\frac{\a{V(C^{\P})}-3}{2}\\
	& = & \a B+\frac{\a{V(C)\u V(C^{\p})}-2}{2}-1\\
	& = & \frac{n-2}{2}.
\end{eqnarray*}

\noindent Thus $S_{1}$ is a maximum independent set and
$\alpha(G)=\frac{n-2}{2}$. Similarly, we may consider the following
maximum independent sets:

\begin{align*}
	S_{2} & =B\u\{c_{i}:i\text{ is odd and }i\neq1\}\u\{c_{i}^{\P}:i\text{ is odd and }i\neq1,t\},\\
	S_{3} & =A\u\{c_{i}^{\P}:i\text{ is even}\}\u\{c_{i}:i\text{ is even and }i\neq2\},\\
	S_{4} & =A\u\{c_{i}^{\P}:i\text{ is odd and }i\neq1\}\u\{c_{i}:i\text{ is odd and }i\neq1,k\},
\end{align*}

\noindent It is easy to see that the union/intersection of the sets
$S_{1},\dots,S_{4}$ yields $V(G)$ and $\emptyset$, respectively.
In other words, $\core G=\emptyset$ and $\corona G=V(G).$
\end{proof}

\begin{corollary}\label{123kj123j}
If $G$ is a bicyclic odd-linked graph of order
$n$, then $\alpha(G)+\mu(G)=n-1$.
\end{corollary}

\begin{corollary}\label{123kj123j123kj123j}
If $G$ is a bicyclic odd-linked graph, then
$\a{\core G}+\a{\corona G}=2\alpha(G)$ and $\corona G$ and $N(\core G)$
form a partition of $V(G)$. 
\end{corollary}

\section{Fused-odd graphs}\label{sss5}

For completeness in this section, we analyze fused-odd graphs; due to the structure of these graphs, the analysis is much simpler than in the previous cases. For the rest, recall the following classical result on factor-critical graphs.

\begin{theorem}
	[\cite{lovasz1972structure}\label{lovasz}] A graph $G$ is factor-critical
	if and only if $G$ has an odd-length ear decomposition starting from
	an odd cycle. 
\end{theorem}

\begin{lemma}\label{lema2}
	A factor-critical graph of order at least two is not a
	K\H{o}nig--Egerv\'ary graph.
\end{lemma}

\begin{proof}
	Let $G$ be a factor-critical graph of order $n$. Suppose
	that $G$ is a K\H{o}nig--Egerv\'ary graph; then
	$\mu(G)=\frac{n-1}{2}$, and hence
	$\alpha(G)=n-\frac{n-1}{2}=\frac{n+1}{2}$.
	Let $S$ be a maximum independent set of $G$ and let $v\notin S$.
	Then $G-v$ has order $n-1$ and an independent set with
	$\frac{n+1}{2}$ vertices; therefore, $G-v$ does not have a
	perfect matching, a contradiction.
\end{proof}

\begin{theorem}\label{io12312}
	Let $G$ be a fused-odd graph of order $n$. Then
	\begin{itemize}
		\item $\mu(G)=\alpha(G)=\frac{n-1}{2}$,
		\item $\core G=\emptyset$, 
		\item $\corona G=V(G)$ if and only if the odd cycles share at least two vertices.
	\end{itemize}
\end{theorem}

\begin{proof}
	By \cref{lovasz}, $G$ is a factor-critical graph, and hence
	$\mu(G)=\frac{n-1}{2}$.
	But by \cref{lema2}, $G$ is not a K\H{o}nig--Egerv\'ary graph; therefore,
	$\alpha(G)+\mu(G)=\alpha(G)+\frac{n-1}{2}\le n-1$, and thus
	$\alpha(G)\le\frac{n-1}{2}$.
	Since $G$ is a fused-odd graph, there is a vertex $x$ such that the graph $G-x$ is a bipartite graph with perfect matching. Then note that $\alpha(G-x)+\mu(G-x)=n-1$, $\mu(G-x)=\mu(G)$
	and $\alpha(G)+\mu(G)<n$. But since $\alpha(G-x)\le\alpha(G)$ we
	have that $\alpha(G-x)=\alpha(G)$. Therefore $\Omega(G-x)\s\Omega(G)$,
	and hence $\core G=\emptyset$. 
	
	Trivially $x$ is unique if and only if the odd cycles share
	exactly one vertex (see \cref{asopdkakso123123}). Otherwise the previous argument
	shows that $\corona G=V(G)$. If $x$ is unique it is easy
	to see that $\corona G=V(G)-\{x\}$.

\end{proof}

\begin{corollary}\label{io12312io12312}
	Let $G$ be a fused-odd graph of order $n$. Then
	\begin{itemize}
		\item $\a{\corona G}+\a{\core G}=2\alpha(G)+1$,
		\item $\alpha(G)+\mu(G)=n-1$.
	\end{itemize}
\end{corollary}

\section{Conclusions}\label{sss6}

In this final section, we summarize the main consequences of the previous
results and collect general statements for $2$-bicritical graphs with at most
two odd cycles.
We also propose some conjectures and open problems motivated by this work.

We now collect the main consequences obtained in the previous sections into a unified statement. Note that \cref{io12312} and \cref{io12312io12312} also hold when $G$ is a one-odd cycle graph. By definition, one-odd cycle, fused-odd, even-linked, and odd-linked graphs are connected graphs. If $G$ is a disconnected $2$-bicritical graph with at most two odd cycles, then by \cref{puyelarpendietnedesc}, $G$ is formed by two components, each of which is an odd cycle. Hence, the disconnected case is trivial, as well as the one-odd cycle case. Therefore, by \cref{123kj123j}, \cref{123kj123j123kj123j}, \cref{123kj123j123kj123j123kj123j}, and \cref{123kj123j123kj123j123kj123j123kj123j}, we obtain the following.

\begin{theorem}
	Let $G$ be a $2$-bicritical graph with two odd cycles, then:
	\begin{itemize}
		\item If $G$ is connected and the two odd cycles share at most one vertex, then
		$\a{\core G}+\a{\corona G}=2\alpha(G)$,
		\item If both odd cycles share at least two vertices, then
		$\a{\core G}+\a{\corona G}=2\alpha(G)+1$,
		\item If $G$ is disconnected, then
		$\a{\core G}+\a{\corona G}=2\alpha(G)+2$.
	\end{itemize}
\end{theorem}

\begin{theorem}
	Let $G$ be a $2$-bicritical graph with at most two odd cycles. Then
	$\corona G\ud N(\core G)=V(G)$ if and only if there do not exist two odd cycles that share at least two vertices.
\end{theorem}

\begin{theorem}
	Let $G$ be a $2$-bicritical graph with at most two odd cycles.
	If $G$ is connected, then
	\[
	\alpha(G)+\mu(G)=n-1.
	\]
	If $G$ is disconnected, then
	\[
	\alpha(G)+\mu(G)=n-2.
	\]
\end{theorem}

\begin{conjecture}
	If $G$ is a $2$-bicritical graph with two odd
	cycles, then $\core G$ and $\corona G$ can be computed in
	polynomial time.
\end{conjecture}

\begin{problem}
	If $G$ is a graph with $k$ odd cycles, explicitly determine
	$\core G$ and $\corona G$ for $k\ge1$.
\end{problem}

\begin{problem}
	If $G$ is a $2$-bicritical graph with $k$ odd cycles,
	explicitly determine $\core G$ and $\corona G$ for $k\ge3$.
\end{problem}

\section*{Acknowledgments}

	This work was partially supported by Universidad Nacional de San Luis, grants PROICO 03-0723 and PROIPRO 03-2923, MATH AmSud, grant 22-MATH-02, Consejo Nacional de Investigaciones
	Cient\'ificas y T\'ecnicas grant PIP 11220220100068CO and Agencia I+D+I grants PICT 2020-00549 and PICT 2020-04064.

	\section*{Declaration of generative AI and AI-assisted technologies in the writing process}
	During the preparation of this work the authors used ChatGPT-3.5 in order to improve the grammar of several paragraphs of the text. After using this service, the authors reviewed and edited the content as needed and take full responsibility for the content of the publication.

\section*{Data availability}

Data sharing not applicable to this article as no datasets were generated or analyzed during the current study.

\section*{Declarations}

\noindent\textbf{Conflict of interest} \ The authors declare that they have no conflict of interest.

\bibliographystyle{apalike}

\bibliography{TAGcitasV2025}

\end{document}